\documentclass{amsart} 

\newtheorem{theorem}{Theorem}[section]
\newtheorem{lemma}[theorem]{Lemma}

\usepackage{graphicx,amssymb}
\usepackage[T1]{fontenc}

\theoremstyle{definition}

\newtheorem{example}[theorem]{Example}

\theoremstyle{remark}

\numberwithin{equation}{section}

\begin{document}

\title[Quasi-alternating links and $Q$-polynomials]
{Quasi-alternating links and $Q$-polynomials}

\author{Masakazu Teragaito}
\address{Department of Mathematics and Mathematics Education, Hiroshima University,
1-1-1 Kagamiyama, Higashi-hiroshima, Japan 739-8524.}
\email{teragai@hiroshima-u.ac.jp}
\thanks{The author was partially supported by Japan Society for the Promotion of Science,
Grant-in-Aid for Scientific Research (C), 25400093.
}%

\subjclass[2010]{Primary 57M25}



\keywords{quasi-alternating link, $Q$-polynomial, determinant}

\begin{abstract}
Qazaqzeh and Chbili showed that
for any quasi-alternating link,
the degree of $Q$-polynomial is less than its determinant.
We give a refinement of their evaluation.
\end{abstract}

\maketitle

\section{Introduction}\label{sec:intro}

The notion of quasi-alternating links was introduced by Ozsv\'{a}th and Szab\'{o}
\cite{OS}, and it is recognized as one of important classes of links in knot theory.
For example, see \cite{CK, G, GW, MO, QC, QQJ, Wi}.
We recall the definition of quasi-alternating links.

The set of $\mathcal{Q}$ of \textit{quasi-alternating links\/}
is the smallest set of links which satisfies the following properties.
\begin{enumerate}
\item The unknot is in $\mathcal{Q}$.
\item Let $L$ be a link whose diagram $D$ has a crossing $c$ such that
\begin{enumerate}
\item both resolutions $L_\infty$ and $L_0$, obtained from $D$ by smoothing
the crossing $c$ as in Figure \ref{fig:qa}, lie in $\mathcal{Q}$; and
\item $\det L=\det L_\infty+\det L_0$.
\end{enumerate}
Then $L$ lies in $\mathcal{Q}$.
\end{enumerate}
Such a crossing $c$ is called a \textit{quasi-alternating crossing}.

\begin{figure}[tb]
\includegraphics*[scale=0.8]{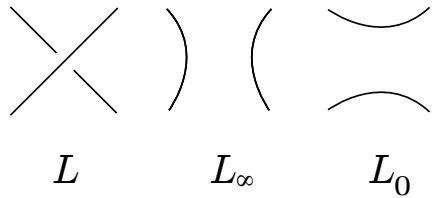}
\caption{Two resolutions of $L$}\label{fig:qa}
\end{figure}

Alternating knots and non-split alternating links are quasi-alternating \cite{OS}.
However, it is not an easy task to determine whether a given knot or link is quasi-alternating or not, in general.
For example, Greene \cite{G} showed that double branched covers do not bound
negative definite $4$-manifolds without homological torsion in order to
prove that the targets are not quasi-alternating.
Also, knot Floer homology and Khovanov homology are known to be an obstruction to 
a link being quasi-alternating \cite{MO}.

On the other hand,
Qazaqzeh and Chbili \cite{QC} found a very simple constraint 
on the highest degree of $Q$-polynomial
for quasi-alternating links.

For unoriented links, $Q$-polynomials were introduced by \cite{BLM} and \cite{Ho}.
Let $L$ be an unoriented link.
Its $Q$-polynomial $Q_L$ is a Laurent polynomial in $\mathbb{Z}[x,x^{-1}]$,
defined as follows.
\begin{enumerate}
\item For the unknot $U$, $Q_U=1$.
\item $Q_{L_+}+Q_{L_-}=x(Q_{L_\infty}+Q_{L_0})$, where
$L_+,L_-,L_\infty,L_0$ are four links
which are identical except in a small region where
they look like as in Figure \ref{fig:q}.
\end{enumerate}

\begin{figure}[tb]
\includegraphics*[scale=0.8]{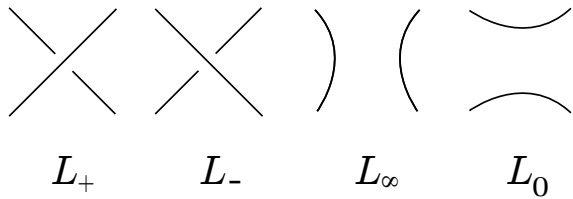}
\caption{Resolutions}\label{fig:q}
\end{figure}

\begin{theorem}[Theorem 1.2 of \cite{QC}]\label{thm:QC}
For any quasi-alternating link $L$,
\[
\deg Q_L\le \det L-1.
\]
\end{theorem} 

For example, the knot $8_{19}$, which is the torus knot of type $(3,4)$,
has determinant $3$, but the degree of its $Q$-polynomial is $6$.
Thus $8_{19}$ is not quasi-alternating.

In general, for any link but the unknot, the degree of its $Q$-polynomial
is less than the crossing number (\cite{BLM}).
And, it is a classical fact that the crossing number is less than or equal to the determinant
for any non-split alternating link (\cite{B,C}).
Thus Theorem  \ref{thm:QC} can be seen as a natural generalization
of the same evaluation for non-split alternating links.

The purpose of this short paper is to give a slight improvement of
the evaluation by Qazaqzeh and Chbili \cite{QC}.

\begin{theorem}\label{thm:main}
Let $L$ be a quasi-alternating link.
If $L$ is not a $(2,n)$-torus link, then
\[
\deg Q_L\le \det L-2.
\]
\end{theorem}

Of course, the $(2,n)$-torus link $L$ is alternating, so quasi-alternating, unless $n=0$.
It has determinant $|n|$, but it is easy to show that $\deg Q_L=|n|-1$.
Thus the conclusion of Theorem \ref{thm:main} does not hold.
Also, the figure-eight knot has determinant $5$, and its $Q$-polynomial
is $2x^3+4x^2-2x-3$.
Since the figure-eight knot is quasi-alternating,
the evaluation of Theorem \ref{thm:main} is sharp.

In the proof of Theorem \ref{thm:main},
the Dehn surgery characterization of the unknot by \cite{KMOS,OS2} plays a key role.

\section{Proof of Theorem \ref{thm:main}}

The next lemma is a key step of the proof of Theorem \ref{thm:QC}.

\begin{lemma}[Lemma 2.2 of \cite{QC}]\label{lem:QC}
Let $L$ be a link.
Then
\[
\deg Q_L\le \max\{\deg Q_{L_0}, \deg Q_{L_\infty}\}+1,
\]
where $L_0$ and $L_\infty$ are the resolutions of $L$ at any crossing.
\end{lemma}

\begin{lemma}\label{lem:123}
Let $L$ be a quasi-alternating link.
If $\det L=1, 2$ or $3$, respectively,
then $L$ is the unknot, the Hopf link or a trefoil, respectively.
\end{lemma}

\begin{proof}
This is found in the proof of Proposition 3.2 in \cite{G}.
\end{proof}

\begin{lemma}\label{lem:det4}
Let $L$ be a quasi-alternating link.
If $\det L=4$, then
$L$ is the $(2,\pm 4)$-torus link or 
$\deg Q_L\le 2$.
\end{lemma}

\begin{proof}
Let $c$ be a quasi-alternating crossing of $L$.
Let $L_0$ and $L_\infty$ be two resolutions of $L$ at the crossing $c$.
Then both of $L_0$ and $L_\infty$ are quasi-alternating.
Since $\det L=\det L_0 +\det L_\infty$,
$\{\det L_0,\det L_\infty\}=\{3,1\}$ or
$\det L_0=\det L_\infty=2$.

First, we may assume that $\det L_0=3$ and $\det L_\infty=1$.
By Lemma \ref{lem:123}, $L_0$ is a trefoil and $L_\infty$ is the unknot.
Let $\gamma$ be an unknotted arc connecting the strands at the resolution of $L_\infty$,
and let $K$ be the lift of $\gamma$ in the double branched cover $\Sigma(L_\infty)=S^3$.
Then $\Sigma(L_0)=\pm L(3,1)$ is obtained by an integral Dehn surgery on $K$.
By \cite[Theorem 1.1 and Corollary 8.4]{KMOS} (or \cite[Theorem 9]{T}) and \cite{M}, $K$ is the unknot.
Hence $\Sigma(L)$ is also obtained by an integral Dehn surgery on the unknot $K$,
so $\Sigma(L)=\pm L(4,1)$.
This implies that $L$ is the $(2,\pm 4)$-torus link by \cite{HR}.

Next, assume that $\det L_0=\det L_\infty=2$.
By Lemma \ref{lem:123} again, both $L_0$ and $L_\infty$ are Hopf links.
Note that $Q_{L_0}=Q_{L_\infty}=2x+1-2x^{-1}$.
By Lemma \ref{lem:QC}, $\deg Q_L\le 2$.
\end{proof}

It seems to be open that a quasi-alternating link with determinant 4
should be either the $(2,\pm 4)$-torus link or the connected sum of two Hopf links.

\begin{proof}[Proof of Theorem \ref{thm:main}]
The argument is done by induction on determinant of $L$.
First, we note that
$\det L\ge 4$  by Lemma \ref{lem:123}, under our assumption.

Suppose that $\det L=4$.
By Lemma \ref{lem:det4} and our assumption,
$\deg Q_L\le 2$.

Now, suppose that the conclusion is true for any quasi-alternating link
with determinant less than or equal to $m\,(\ge 4)$, which is not a $(2,n)$-torus link.
Let $L$ be a quasi-alternating link with determinant $m+1$.
Choose a quasi-alternating crossing $c$, and let $L_0$ and $L_\infty$
be the resolutions at $c$.
Then both of the resolutions are quasi-alternating, and 
the equation $\det L=\det L_0+\det L_\infty$ holds.
Thus $L_0$ and $L_\infty$ have determinant less than or equal to $m$.

We split the argument into 3 cases.

\medskip

(1)  Neither $L_0$ nor $L_\infty$ is a $(2,n)$-torus link.

By inductive hypothesis, we have $\deg Q_{L_0}\le \det L_0-2$
and $\deg Q_{L_\infty}\le \det L_\infty-2$.
Thus,
\begin{eqnarray*}
\deg Q_L &\le& \max\{\deg Q_{L_0}, \deg Q_{L_\infty}\}+1\\
&=& \deg Q_{L_*}+1\\
&\le & (\det L_*-2)+1\\
&\le & \det L-2,
\end{eqnarray*}
where $*\in \{0,\infty\}$ is chosen appropriately.
The last inequality follows from the equation $\det L=\det L_0+\det L_\infty$.

\medskip

(2) Only one of $L_0$ and $L_\infty$ is a $(2,n)$-torus link.

We may assume that $L_0$ is the $(2,p)$-torus link.
Then $\det L_0=|p|$ and $\deg Q_{L_0}=|p|-1$.
For $L_\infty$, we have
$\deg Q_{L_\infty}\le \det L_\infty-2$ by inductive hypothesis.
If $\deg Q_{L_0}\le \deg Q_{L_\infty}$, then
$\deg Q_L\le (\det L_\infty-2)+1\le \det L-2$ as in (1).

Otherwise, we have $\deg Q_L\le \deg Q_{L_0}+1=|p|$.
Since $\det L_\infty\ge 4$ by Lemma \ref{lem:123},
we have $|p|=\det L_0\le \det L-4$.
Hence $\deg Q_L\le \det L-4$.

\medskip

(3) Both of $L_0$ and $L_\infty$ are $(2,n)$-torus links.

We assume that $L_\infty$ is the $(2,p)$-torus link, and $L_0$ is the $(2,q)$-torus link.
Moreover, we may assume that $|p|\le |q|$.
Then
\[
\deg Q_L\le \max\{\deg Q_{L_0}, \deg Q_{L_\infty}\}+1=|q|=\det L-|p|.
\]
Since $L_\infty$ is quasi-alternating,
$p\ne 0$.
If $|p|\ne 1$, then we have $\deg Q_L\le \det L-2$.
If $|p|=1$, then $L_\infty$ is the unknot.
Then, as in the proof of Lemma \ref{lem:det4},
take an unknotted arc $\gamma$ connecting the strands at the resolution of $L_\infty$.
Let $K$ be the lift of $\gamma$ in $\Sigma(L_\infty)=S^3$.
Then an integral Dehn surgery on $K$ yields $\Sigma(L_0)=L(q,1)$.
We remark that $|q|=\det L-1=m\ge 4$.
By \cite[Theorem 1.1]{KMOS} and \cite[Theorem 9]{T}, $K$ is the unknot,
or a trefoil.

Assume that $K$ is the unknot.
Then, since
$\Sigma(L)$ is obtained by an integral Dehn surgery on $K$,
it is a lens space $\pm L(r,1)$, with $r=\det L$.
But this implies that $L$ is the $(2,\pm r)$-torus link by \cite{HR}, a contradiction.

Finally, assume that $K$ is a trefoil.
By \cite[Theorem 9]{T}, $|q|=5$. 
Thus we have $\det L=6$.
For a trefoil $K$ in $\Sigma(L_\infty)$, it is well known that
there is the unique inverting involution (see \cite{S}).
We may assume that $K$ is right-handed.
By taking the quotient of $(\Sigma(L_\infty),K)$ under the involution,
we can recover $\gamma$ as in Figure \ref{fig:gamma}, with
ignoring the framing of $\gamma$.

\begin{figure}[tb]
\includegraphics*[scale=0.8]{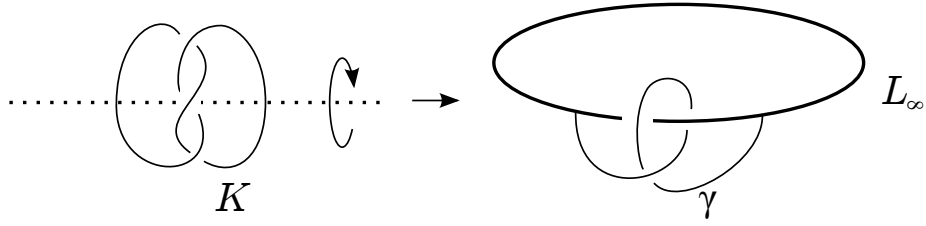}
\caption{$(L_\infty,\gamma)$}\label{fig:gamma}
\end{figure}

Since $L_0$ is obtained from $L_\infty$ by banding along $\gamma$,
we have Figure \ref{fig:L0}, where $k$ denotes
the number of half-twists.
(If $k\ge 0$, then the twists are right-handed.
Otherwise, left-handed.)

\begin{figure}[tb]
\includegraphics*[scale=0.8]{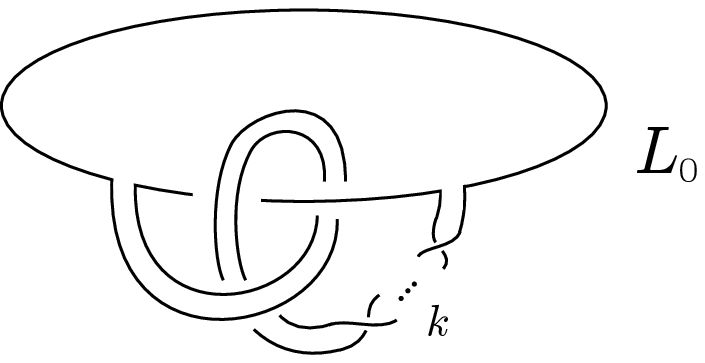}
\caption{$L_0$}\label{fig:L0}
\end{figure}

Then $L_0$ is the pretzel link of type $(2,-3,k-2)$.
Since $L_0$ is $2$-bridge, $|k-2|\le 1$.
Hence $k=1,2$ or $3$.
The only possibility for $L_0$ to be the $(2,\pm 5)$-torus knot
is $k=1$.
Then $L$ is the pretzel link of type $(2,3,0)$ or $(2,3,-2)$.
The former gives the connected sum of a trefoil and the Hopf link,
so $\deg Q_L=3$.
Hence $\deg Q_L=\det L-3$.
The latter has determinant $4$, contradicting $\det L=6$.
\end{proof}

\begin{example}
Let $K$ be the knot $10_{140}$ in the knot table.
It is hyperbolic and has determinant $9$.
But the $Q$-polynomial is 
$2x^8+4x^7-12x^6-22x^5+24x^4+32x^3-24x^2-12x+9$, so
$K$ is not quasi-alternating by Theorem \ref{thm:main}.
The evaluation (Theorem \ref{thm:QC}) of Qazaqzeh and Chbili \cite{QC}
cannot detect this fact.
We remark that this knot is known to be non-quasi-alternating, because
it has thick odd Khovanov homology (see \cite[p.2456]{CK}).

Also, among $11$, $12$-crossing non-alternating knots expressed in Dowker-Thistlethwaite notation,
$12_{n0025}$ ,$12_{n0093}$,$12_{n0115}$,$12_{n0138}$,
$12_{n0199}$,$12_{n0321}$,$12_{n0355}$,
$12_{n0374}$,$12_{n0433}$,\\
$12_{n0457}$,$12_{n0648}$
have determinant $11$, but the degree of their $Q$-polynomials is $10$
(see \cite{CL}).
Thus these are not quasi-alternating.
Again, this fact was confirmed in \cite{J} by using homologically thickness.
\end{example}

\bibliographystyle{amsplain}

\end{document}